\documentclass[11pt]{article}%
\usepackage[letterpaper, portrait, margin=1in, headheight=0pt]{geometry}
\usepackage{mathtools}
\usepackage{enumerate}
\usepackage{amsmath,amsthm,amssymb,amsfonts}
\usepackage{enumitem}
\usepackage{mathrsfs}
\usepackage{tikz-cd}
\usepackage{multicol}
\usepackage[title]{appendix}
\usetikzlibrary{quotes,angles}

\DeclareMathOperator{\As}{As}
\DeclareMathOperator{\Dias}{Dias}

\DeclareMathOperator{\Hom}{Hom}

\DeclareMathOperator{\Inf}{Inf}
\DeclareMathOperator{\Tra}{Tra}
\DeclareMathOperator{\Res}{Res}
\DeclareMathOperator{\ima}{Im}

\newcommand{\N}{\mathbb{N}}

\newcommand{\F}{\mathbb{F}}

\newcommand{\cZ}{\mathcal{Z}}
\newcommand{\cB}{\mathcal{B}}
\newcommand{\cH}{\mathcal{H}}

\newcommand{\vv}{_{\vdash}}
\newcommand{\dd}{_{\dashv}}

\newtheorem{thm}{Theorem}[section]
\newtheorem{lem}[thm]{Lemma}
\newtheorem{prop}[thm]{Proposition}
\newtheorem{cor}[thm]{Corollary}

\theoremstyle{definition}
\newtheorem{defn}{Definition}

\theoremstyle{remark}
\newtheorem*{rem}{Remark}
\title{Multipliers of Nilpotent Diassociative Algebras}
\author{Erik Mainellis}
\date{}

\begin{document}

\maketitle

\begin{abstract}
    The paper concerns nilpotent associative dialgebras and their corresponding diassociative Schur multipliers. Using Lie (and group) theory as a guide, we first extend a classic five-term cohomological sequence under alternative conditions in the nilpotent setting. This main result is then applied to obtain a new proof for a previous extension of the same sequence. It also yields a different extension of the sequence that involves terms in the upper central series. Furthermore, we use the main result to obtain a collection of dimension bounds on the multiplier of a nilpotent diassociative algebra. These differ notably from the Lie case. Since diassociative algebras generalize associative algebras, we obtain an associative analogue of the results herein. We conclude by computing both the associative and diassociative multipliers of an associative algebra. This paper is part of an ongoing project to advance extension theory in the context of several Loday algebras.
\end{abstract}

\section{Introduction}
In 1904, Schur introduced his multiplier in the context of group representation theory \cite{schur}. A comprehensive investigation of \textit{the Schur multiplier} can be found in Karpilovsky's book of the same name \cite{karp}. Beyond the context of groups, notions of multipliers and their surrounding theory have since been developed for Lie algebras \cite{batten}, for Leibniz algebras \cite{mainellis batten, rogers}, and, most recently, for diassociative algebras \cite{mainellis batten di, mainellis perfect}. The aim of the present paper is to investigate the multipliers of nilpotent diassociative algebras. Diassociative algebras, also known as associative dialgebras, were introduced by Loday in \cite{loday dialgebras}. Arising in the context of algebraic $K$-theory, they have since been found to have connections with algebraic topology, among other fields. Formally, and over a field $\F$, a \textit{diassociative algebra} $L$ is an $\F$-vector space equipped with two associative bilinear products, denoted $\dashv$ and $\vdash$, that satisfy
\begin{align}
    x\dashv (y\dashv z) = x\dashv (y\vdash z), \\
    (x\vdash y)\dashv z = x\vdash (y\dashv z), \\
    (x\dashv y)\vdash z = (x\vdash y)\vdash z~
\end{align}
for all $x,y,z\in L$. Part of the motivation for studying the multiplier of these algebras is their classification. In Lie theory, for example, classes of nilpotent algebras have been characterized by an invariant related to the dimension of their multiplier (see \cite{hardy early, hardy}). We refer the reader to \cite{basri} and the references therein for some classification of diassociative algebras and their connections with other classes of Loday algebras. It is remarkable that the theory contained in the present paper holds in the diassociative setting, despite these algebras having a significantly different structure than Lie algebras.

Given a diassociative algebra $L$, its multiplier is defined as follows.

\begin{defn}
A \textit{definining pair} $(K,M)$ of $L$ is itself a pair of diassociative algebras that satisfies $K/M\cong L$ and $M\subseteq Z(K)\cap K'$. Such a pair is called a \textit{maximal defining pair} if the dimension of $K$ is maximal. In this case, we say that $K$ is a \textit{cover} of $L$ and that $M$ is the \textit{multiplier} of $L$, denoted by $M(L)$.
\end{defn}

The necessary theory for this paper has been established in \cite{mainellis batten di}. Therein, it is shown that the multiplier of a diassociative algebra is characterized by its second cohomology group with coefficients in the field, i.e. that $M(L)\cong \cH^2(L,\F)$. Moreover, the following exact five-term cohomological sequence\footnote{Based on the central extension $0\xrightarrow{} Z\xrightarrow{} L\xrightarrow{} L/Z\xrightarrow{} 0$.} \[0\xrightarrow{} \Hom(L/Z,\F)\xrightarrow{\Inf_1} \Hom(L,\F)\xrightarrow{\Res} \Hom(Z,\F)\xrightarrow{\Tra} \cH^2(L/Z,\F)\xrightarrow{\Inf_2} \cH^2(L,\F)\] is obtained, where $Z$ is a central ideal of $L$ and $\F$ is seen as a central $L$-module. The sequence is also extended by $\delta:\cH^2(L,\F)\xrightarrow{} (L/L'\otimes Z \oplus Z\otimes L/L')^2$.

The present paper is structured as follows. Following (primarily) a similar methodology to \cite{yankosky}, we begin by extending the aforementioned five-sequence under alternative conditions. This easily yields a new proof for the original extension in the nilpotent case. As further applications, we obtain a handful of dimension bounds on the multiplier of a nilpotent diassociative algebra as well as another extension \[\cH^2(L/L^n,\F)\xrightarrow{\Inf} \cH^2(L,\F) \xrightarrow{\delta} (L/Z_{n-1}\otimes L^n\oplus L^n\otimes Z_{n-1})^2\] of the five sequence, where $Z_j$ denotes the $j$th term in the upper cental series of $L$, $L^j$ denotes the $j$th term in the lower central series of $L$, and $L$ is nilpotent of class $n$. Notably, diassociative algebras generalize associative algebras, and we thus obtain associative analogues of these results. It is particularly interesting to consider an associative algebra and compare its multiplier \textit{as} an associative algebra to its multiplier \textit{as} a diassociative algebra. Such a phenomenon has been explored in the context of Lie and Leibniz multipliers \cite{rogers}. We compute an example that highlights the associative to diassociative comparison as well as a couple of our dimension bounds.

\section{Preliminaries}
We use the following definitions throughout. Given a diassociative algebra $L$ with multiplications $\dashv$ and $\vdash$, a subspace $A$ of $L$ is a \textit{subalgebra} if $a\dashv b$, $a\vdash b\in A$ for all $a,b\in A$. A subalgebra $I$ of $L$ is called an \textit{ideal} if $x\dashv i$, $x\vdash i$, $i\dashv x$, $i\vdash x\in I$ for all $i\in I$ and $x\in L$. The \textit{center} $Z(L)$ of $L$ is the ideal consisting of all elements $z\in L$ such that $z\dashv x = z\vdash x = x\dashv z = x\vdash z = 0$ for all $x\in L$. Consider two subalgebras $A$ and $B$ in $L$. We denote $A\lozenge B = A\dashv B + A\vdash B$ and say that $L$ is \textit{perfect} if $L= L'$, where $L' = L\lozenge L$.

The definition of nilpotency for diassociative algebras is naturally more complex than that of algebras with a single multiplication, although it is also more flexible in certain ways. As with associative algebras, there is no non-nilpotent solvable diassociative algebra. We begin with the following sequences of ideals for a diassociative algebra $L$, as constructed in \cite{basri}, that are shown to be equivalent notions in the same work. We state this latter fact as a lemma.

\begin{enumerate}
    \item[i.] $L^{<1>} = L$, $L^{<n+1>} = L^{<n>}\lozenge L$,
    \item[ii.] $L^{\{1\}} = L$, $L^{\{n+1\}} = L\lozenge L^{\{n\}}$,
    \item[iii.] $L^1 = L$, $L^{n+1} = L^1\lozenge L^n + L^2\lozenge L^{n-1} + \cdots + L^n\lozenge L^1$.
\end{enumerate}

\begin{lem}\label{nilpotent equivalence}
For a diassociative algebra $L$, one has $L^n = L^{\{n\}} = L^{<n>}$ for all $n\in \N$.
\end{lem}

We now narrow our focus to $L^n$, which we will define to be the $n$th term for the \textit{lower central series} of $L$, as the most natural-feeling collection of all $n$ products, although we will still use the flexibility afforded by Lemma \ref{nilpotent equivalence} in this paper. A diassociative algebra $L$ that satisfies $L^n\neq 0$ and $L^{n+1}=0$ is called \textit{nilpotent of class $n$}. If $L$ is nilpotent of class $n$, it induces a central extension $0\xrightarrow{} L^n\xrightarrow{} L\xrightarrow{} L/L^n\xrightarrow{} 0$. We also define the \textit{upper central series} of $L$ by $Z_1 = Z(L)$ and \[Z_{j+1} = \{x\in L~|~ \forall l\in L,~ x\dashv l, x\vdash l, l\dashv x, l\vdash x\in Z_{j}\}\] for $j\geq 1$. In particular, for a nilpotent diassociative algebra $L$ of class $n$, one has $Z_n = L$, and thus $L'\subseteq Z_{n-1}$.

\begin{lem}\label{upper central lemma}
Let $L$ be a diassociative algebra and $Z_j$ denote the $j$th term in the upper central series of $L$. Then \[L^s\lozenge Z_i + Z_i\lozenge L^s\subseteq Z_{i-s}\] for all $i\geq s$.
\end{lem}

\begin{proof}
We proceed by induction on $s$. For the base case $s=1$, one has $L\lozenge Z_i + Z_i\lozenge L\subseteq Z_{i-1}$ by definition. Now assume that the statement holds for some $s\geq 1$. We compute \begin{align*}
    L^{s+1}\lozenge Z_i + Z_i\lozenge L^{s+1} &= (L\lozenge L^s)\lozenge Z_i + Z_i\lozenge (L^s\lozenge L) & \text{by Lemma \ref{nilpotent equivalence}} \\ &= L\lozenge (L^s\lozenge Z_i) + (Z_i\lozenge L^s)\lozenge L & \text{$(*)$} \\ &\subseteq L\lozenge Z_{i-s} + Z_{i-s}\lozenge L & \text{by induction}\\ &\subseteq Z_{i-s-1} & \text{by definition}
\end{align*} where $(*)$ follows via the diassociative identities of $L$ and the associativity of its multiplications.
\end{proof}

As discussed in \cite{mainellis batten di}, given a pair of diassociative algebras $A$ and $B$, and a central extension $0\xrightarrow{} A\xrightarrow{} L\xrightarrow{} B\xrightarrow{} 0$ of $A$ by $B$, a 2-cocycle $(f\dd,f\vv)\in \cZ^2(B,A)$ is a pair of bilinear forms $f\dd,f\vv:B\times B\xrightarrow{} A$ that satisfy
\begin{enumerate}
    \item[C1.] $f\dd(i,j\dashv k) = f\dd(i,j\vdash k)$,
    \item[C2.] $f\dd(i\vdash j,k) = f\vv(i,j\dashv k)$,
    \item[C3.] $f\vv(i\dashv j,k) = f\vv(i\vdash j,k)$,
    \item[C4.] $f\dd(i,j\dashv k) = f\dd(i\dashv j,k)$,
    \item[C5.] $f\vv(i,j\vdash k) = f\vv(i\vdash j,k)$
\end{enumerate} for $i,j,k\in B$. Throughout, we will refer to C1 - C5 as the diassociative cocycle conditions or cocycle identities. See \cite{mainellis factor systems} for more on 2-cocycles and their relations with extension theory.

\section{The Main Result}
Throughout, the form $(X\otimes Y\oplus Y\otimes X)^2$ will denote $(X\otimes Y\oplus Y\otimes X) \oplus (X\otimes Y\oplus Y\otimes X)$. This is not to be confused with the terms $L^n$ in the lower central series of an algebra $L$. We will use $L'$ to denote the second term $L^2$ in the lower central series.

\begin{thm}\label{yankosky 2.1}
Let $L$ be a nilpotent diassociative algebra and let $A$ and $B$ be ideals in $L$ such that $L'\subseteq A$ and $B\subseteq Z(L)$. If $f\dd(A,B)=0$, $f\dd(B,A) = 0$, $f\vv(A,B)=0$, and $f\vv(B,A) = 0$ for all $(f\dd,f\vv)\in \cZ^2(L,\F)$, then there exists a homomorphism $\delta$ such that \[\cH^2(L/B,\F) \xrightarrow{\Inf} \cH^2(L,\F)\xrightarrow{\delta} (L/A\otimes B\oplus B\otimes L/A)^2\] is exact.
\end{thm}

\begin{proof}
Consider elements $x\in L$, $b\in B$, and $(f\dd',f\vv')\in \cZ^2(L,\F)$. We define the bilinear forms \begin{align*}
    f\dd'':L/A\times B\xrightarrow{} \F, && g\dd'':B\times L/A\xrightarrow{} \F, \\ f\vv'':L/A\times B\xrightarrow{} \F, && g\vv'':B\times L/A\xrightarrow{} \F
\end{align*} by \begin{align*}
    f\dd''(x+A,b) = f\dd'(x,b), && g\dd''(b,x+A) = f\dd'(b,x), \\ f\vv''(x+A,b) = f\vv'(x,b), && g\vv''(b,x+A) = f\vv'(b,x).
\end{align*} Since $f\dd(A,B)=0$, $f\dd(B,A) = 0$, $f\vv(A,B)=0$, and $f\vv(B,A) = 0$, all of these maps are well-defined. We define $\delta':\cZ^2(L,\F)\xrightarrow{} (L/A\otimes B\oplus B\otimes L/A)^2$ by $\delta'(f\dd',f\vv') = (f\dd'',g\dd'',f\vv'',g\vv'')$. Now consider an element $(f\dd',f\vv')\in \cB^2(L,\F)$. Then there exists a linear transformation $\varepsilon:L\xrightarrow{} \F$ such that $f\dd'(x,y) = -\varepsilon(x\dashv y)$ and $f\vv'(x,y) = -\varepsilon(x\vdash y)$ for all $x,y\in L$. For $b\in B$, however, we compute \begin{align*}
    f\dd''(x+A,b) = f\dd'(x,b) = -\varepsilon(x\dashv b) = 0, && g\dd''(b,x+A) = f\dd'(b,x) = -\varepsilon(b\dashv x) = 0, \\ f\vv''(x+A,b) = f\vv'(x,b) = -\varepsilon(x\vdash b) = 0, && g\vv''(b,x+A) = f\vv'(b,x) = -\varepsilon(b\vdash x) = 0
\end{align*} since $B\subseteq Z(L)$. Thus, $\delta'(f\dd',f\vv') = 0$ for any coboundary $(f\dd', f\vv')$, and so $\delta'$ induces a well-defined map $\delta:\cH^2(L,\F)\xrightarrow{} (L/A\otimes B\oplus B\otimes L/A)^2$ by $\delta((f\dd',f\vv')+\cB^2(L,\F)) = (f\dd'',g\dd'',f\vv'',g\vv'')$.

Now that we have established our $\delta$, it remains to show that the sequence is exact. Consider a cocycle $(f\dd,f\vv)\in \cZ^2(L/B,\F)$ and set $f\dd'(x,y) = f\dd(x+B,y+B)$ and $f\vv'(x,y) = f\vv(x+B,y+B)$. We first recall (see \cite{mainellis batten di}) that $\Inf:\cH^2(L/B,\F)\xrightarrow{} \cH^2(L,\F)$ is defined by $\Inf((f\dd,f\vv)+\cB^2(L/B,\F)) = (f\dd',f\vv') + \cB^2(L,\F)$. To show that $\ima(\Inf)\subseteq \ker \delta$, consider $(f\dd,f\vv)\in \cZ^2(L/B,\F)$. Then $(f\dd,f\vv)$ induces tuples $(f\dd',f\vv')$ and $(f\dd'',g\dd'',f\vv'',g\vv'')$, as defined previously. For $x\in L$ and $b\in B$, one computes \begin{align*}
    &f\dd''(x+A, b) = f\dd'(x,b) = f\dd(x+B,b+B) = 0, \\ & g\dd''(b,x+A) = f\dd'(b,x) = f\dd(b+B,x+B) = 0, \\ &f\vv''(x+A, b) = f\vv'(x,b) = f\vv(x+B,b+B) = 0, \\ &g\vv''(b,x+A) = f\vv'(b,x) = f\vv(b+B,x+B) = 0
\end{align*} and so $\delta(\Inf((f\dd,f\vv)+\cB^2(L/B,\F))) = \delta((f\dd',f\vv')+\cB^2(L,\F)) = (f\dd'',g\dd'',f\vv'',g\vv'') = (0,0,0,0)$. Thus, $\ima(\Inf)\subseteq \ker \delta$.

Conversely, consider a cocycle $(f\dd',f\vv')\in \cZ^2(L,\F)$ such that $(f\dd',f\vv')+\cB^2(L,\F)\in \ker \delta$. In other words, $\delta((f\dd',f\vv')+\cB^2(L,\F)) = (f\dd'',g\dd'',f\vv'',g\vv'') = (0,0,0,0)$, where $f\dd''$, $g\dd''$, $f\vv''$, and $g\vv''$ are defined using $(f\dd',f\vv')$ as above. This implies that \begin{align*}
    f\dd'(x,b) = f\dd''(x+A, b) = 0, && f\dd'(b,x) = g\dd''(b,x+A) = 0, \\
    f\vv'(x,b) = f\vv''(x+A,b) = 0, && f\vv'(b,x) = g\vv''(b,x+A) = 0
\end{align*} for all $x\in L$, $b\in B$. Define a pair $(f\dd,f\vv)$ of bilinear forms $L/B\times L/B\xrightarrow{} \F$ by $f\dd(x+B,y+B) = f\dd'(x,y)$ and $f\vv(x+B,y+B) = f\vv'(x,y)$ for any $x,y\in L$. To show that $f\dd$ and $f\vv$ are well-defined, consider elements $x+B = x_1+B$ and $y+B = y_1+B$. Then $x_1 = x+b$ and $y_1 = y+c$ for some $b,c\in B$. We compute \begin{align*}
    f\dd(x_1+B,y_1+B) &= f\dd'(x_1,y_1) \\ &= f\dd'(x+b,y+c) \\ &= f\dd'(x,y) + f\dd'(x,c) + f\dd'(b,y) + f\dd'(b,c) \\ &= f\dd'(x,y) \\ &= f\dd(x+B,y+B)
\end{align*} and, similarly, $f\vv(x_1+B,y_1+B) = f\vv(x+B,y+B)$. Moreover, $(f\dd,f\vv)$ satisfies the diassociative cocycle conditions since $(f\dd',f\vv')$ does. We thus obtain an element $(f\dd,f\vv)\in \cZ^2(L/B,\F)$ such that $\Inf((f\dd,f\vv) + \cB^2(L/B,\F)) = (f\dd',f\vv') + \cB^2(L,\F)$. Therefore, $\ker \delta\subseteq \ima(\Inf)$, and the sequence is exact.
\end{proof}

\section{Applications}
Our first application of Theorem \ref{yankosky 2.1} is an alternative proof of Theorem 5.1\footnote{Theorem 4.1 in arXiv version.} from \cite{mainellis batten di} in the case when $L$ is a nilpotent diassociative algebra. Letting $A=L'$ and $B=Z\subseteq Z(L)$, we obtain the following.

\begin{prop}\label{yankosky 3.1}
Let $L$ be a nilpotent diassociative algebra and $Z$ be a central ideal in $L$. Then \[\cH^2(L/Z,\F)\xrightarrow{\Inf} \cH^2(L,\F)\xrightarrow{\delta} (L/L'\otimes Z\oplus Z\otimes L/L')^2\] is exact.
\end{prop}

\begin{proof}
To invoke Theorem \ref{yankosky 2.1}, it suffices to show that $f\dd(L',Z)=0$, $f\dd(Z,L')=0$, $f\vv(L',Z)=0$, and $f\vv(Z,L')=0$ for all $(f\dd,f\vv)\in \cZ^2(L,\F)$. But this holds by the diassociative cocycle identities and their ability to associate products within the bilinear forms. For example, we compute \begin{align*}
    f\dd(L',Z) &= f\dd(L\dashv L, Z) + f\dd(L\vdash L,Z) \\ &= f\dd(L,L\dashv Z) + f\vv(L,L\dashv Z) \\ &= 0
\end{align*} via C4 and C2 respectively.
\end{proof}

The following corollary is the diassociative analogue of a result that was proved in \cite{hardy early}, and we use a similar approach to our proof.

\begin{cor}\label{yankosky 3.2}
Let $L$ be a nilpotent, finite-dimensional diassociative algebra and $Z\subseteq Z(L)\cap L'$ be an ideal such that $\dim Z = 1$. Then \[\dim \cH^2(L,\F) + 1\leq \dim \cH^2(L/Z,\F) + 4\dim(L/L').\]
\end{cor}

\begin{proof}
We may invoke our extended cohomological five-sequence regardless of $\dim Z$. \[0\xrightarrow{} \Hom(L/Z,\F)\xrightarrow{} \Hom(L,\F)\xrightarrow{\Res} \Hom(Z,\F)\xrightarrow{\Tra} \cH^2(L/Z,\F) ~~~~~~~~~~~~~~~~~~~~~~~~~~~~~\] \[~~~~~~~~~~~~~~~~~~~~~~~~~~~~~~~~~~~~~~~~~~~~~~\xrightarrow{\Inf} \cH^2(L,\F) \xrightarrow{\delta} (L/L'\otimes Z\oplus Z\otimes L/L')^2\] Since $Z\subseteq L'$, we obtain $\Res = 0$. This follows since $\Res$ simply restricts any homomorphism $\pi:L\xrightarrow{} \F$ to $\pi|_Z$, and any product in $\F$ is zero. By exactness, $\Tra$ is injective. Thus, $\dim(\ima(\Tra)) = 1$ since $\dim(\Hom(Z,\F)) = 1$. Also by exactness, we know that \[\dim(\ima \delta) + \dim(\ima(\Inf)) = \dim \cH^2(L,\F)\] and \[\dim(\ima(\Inf)) + \dim(\ima(\Tra)) = \dim \cH^2(L/Z,\F).\] We therefore compute \begin{align*}
    \dim \cH^2(L,\F) + 1 &= \dim \cH^2(L,\F) + \dim(\ima(\Tra)) \\ &= \dim(\ima \delta) + \dim(\ima(\Inf)) + \dim(\ima(\Tra)) \\ &= \dim(\ima \delta) + \dim \cH^2(L/Z,\F) \\ &\leq \dim((L/L'\otimes Z\oplus Z\otimes L/L')^2) + \dim \cH^2(L/Z,\F) \\ &= 4\dim(L/L') + \dim \cH^2(L/Z,\F).
\end{align*}
\end{proof}

\begin{thm}\label{yankosky 3.3}
Let $L$ be a nilpotent diassociative algebra of class $n$. Then \[\cH^2(L/L^n,\F)\xrightarrow{\Inf} \cH^2(L,\F) \xrightarrow{\delta} (L/Z_{n-1}\otimes L^n\oplus L^n\otimes Z_{n-1})^2\] is exact.
\end{thm}

\begin{proof}
Supposing that $L$ is nilpotent of class $n$, we first note that $L^n\subseteq Z(L)$ and $L'\subseteq Z_{n-1}$ via our preliminary discussion. To invoke Theorem \ref{yankosky 2.1}, it suffices to show that $f\dd(Z_{n-1},L^n) = f\dd(L^n,Z_{n-1}) = f\vv(Z_{n-1},L^n) = f\vv(L^n,Z_{n-1}) = 0$ for all $(f\dd,f\vv)\in \cZ^2(L,\F)$ and $n\geq 1$. For $n=1$, we have $f\dd(Z_0,L) = f\dd(0,L) = 0$ and, similarly, $f\dd(L,Z_0) = f\vv(Z_0,L) = f\vv(L,Z_0)=0$. For $k\geq 1$, we compute \begin{align*}
    f\dd(Z_k,L^{k+1}) &= f\dd(Z_k,L^k\lozenge L) \\ &= f\dd(Z_k,L^k\dashv L) + f\dd(Z_k,L^k\vdash L) \\ &= f\dd(Z_k\dashv L^k,L) + f\dd(Z_k,L^k\dashv L) \\ & \subseteq f\dd(Z_0,L) + f\dd(Z_k\dashv L^k,L) \\ & \subseteq f\dd(Z_0,L) \\ &= 0
\end{align*} via Lemma \ref{nilpotent equivalence}, C4, C1, and Lemma \ref{upper central lemma}. The other computations follow similarly, and thus the result holds by Theorem \ref{yankosky 2.1}.
\end{proof}

\begin{rem}
The Lie analogue of Theorem \ref{yankosky 3.3} relies on induction, but we note that the diassociative version is attainable without it. This reveals that, while the diassociative cocycle identities lack an anticommutative-type property, they are actually more powerful than the Lie conditions in some ways.
\end{rem}

\begin{cor}\label{yankosky 3.4}
Let $L$ be a nilpotent, finite-dimensional diassociative algebra of class $n$. Then \[\dim \cH^2(L,\F) \leq \dim \cH^2(L/L^n,\F) + 4\dim(L^n)\dim(L/Z_{n-1}) - \dim(L^n).\]
\end{cor}

\begin{proof}
Consider the terms in our extended sequence \[\Hom(L,\F)\xrightarrow{\Res} \Hom(L^n,\F)\xrightarrow{\Tra} \cH^2(L/L^n,\F)\xrightarrow{\Inf} \cH^2(L,\F)\xrightarrow{\delta} (L/Z_{n-1}\otimes L^n\oplus L^n\otimes L/Z_{n-1})^2\] and denote \begin{align*}
    &q= \dim \Hom(L^n,\F), \\ &r= \dim \cH^2(L/L^n,\F), \\ &s=\dim \cH^2(L,\F), \\ &t=\dim((L/Z_{n-1}\otimes L^n\oplus L^n\otimes L/Z_{n-1})^2).
\end{align*} We first note that $\F^n = 0$ for $n\geq 2$, and thus any homomorphism $f:L^n\xrightarrow{} \F^n$, as the restriction of some $f\in \Hom(L,\F)$, is the zero map. Therefore $\Res = 0$, and so $\Tra$ is injective. It follows that $q=\dim(\ima(\Tra)) = \dim(\ker(\Inf))$, which implies that $r-q = \dim \cH^2(L/L^n,\F) - \dim(\ker(\Inf)) = \dim(\ima(\Inf)) \leq s$. On the other hand, we know that $\dim \cH^2(L,\F) - \dim(\ker \delta) = \dim(\ima \delta)\leq t$, and so $s- \dim(\ker \delta) \leq t$. Finally, the equality $r-q = \dim(\ima(\Inf)) = \dim(\ker \delta)$ yields $s-(r-q) \leq t$. We thus obtain \begin{align*}
    \dim \cH^2(L,\F) & \leq \dim \cH^2(L/L^n,\F) + \dim((L/Z_{n-1}\otimes L^n\oplus L^n\otimes L/Z_{n-1})^2) - \dim \Hom(L^n,\F) \\ &= \dim \cH^2(L/L^n,\F) + 4\dim(L/Z_{n-1})\dim(L^n) - \dim(L^n)
\end{align*} from $s\leq t+r-q$.
\end{proof}

\begin{cor}\label{yankosky 3.5}
Let $L$ be a nilpotent, finite-dimensional diassociative algebra. Then \[\dim \cH^2(L,\F) \leq \dim \cH^2(L/L',\F) + \dim(L')[4\dim(L/Z(L)) - 4\dim((L/Z(L))') - 1].\]
\end{cor}

\begin{proof}
We proceed by induction on the nilpotency class of $L$. As a base case, if $L$ is nilpotent of class $1$, then $L'=0$ and the result holds trivially. Suppose now that the result holds for all nilpotent diassociative algebras of class less than $n$. We note the following facts:
\begin{enumerate}
    \item $L/L^n$ is nilpotent of class $n-1$,
    \item $L^n\subseteq Z(L)$,
    \item $L'\subseteq Z_{n-1}(L)$,
    \item $(L/L^n)' = L'/L^n$,
    \item $Z(L)/L^n \subseteq Z(L/L^n)$.
\end{enumerate} Denote $A = (L/L^n)/Z(L/L^n)$ and $B= L/Z(L) = (L/L^n)/(Z(L)/L^n)$. By fact 5, $A$ is a homomorphic image of $B$, and so $\dim(A/A')\leq \dim(B/B')$. We thus have \begin{align*}
    \dim \cH^2(L/L^n,\F) &\leq \dim \cH^2((L/L^n)/(L/L^n)',\F) + \dim((L/L^n)')[4\dim(A/A') - 1] \\ &\leq \dim \cH^2(L/L',\F) + \dim(L'/L^n)[4\dim(B/B')-1]
\end{align*} by induction, and \begin{align*}
    \dim \cH^2(L,\F) \leq \dim \cH^2(L/L^n,\F) + 4\dim(L/Z_{n-1})\dim(L^n) - \dim(L^n)
\end{align*} by Corollary \ref{yankosky 3.4}. Furthermore, \begin{align*}
    \dim(L/Z_{n-1}) \leq \dim(L/(L'+Z(L))) = \dim(B/B')
\end{align*} since $L'+Z(L) \subseteq Z_{n-1}$. Combining these inequalities, we compute \begin{align*}
    \dim \cH^2(L,\F) &\leq \dim \cH^2(L/L^n,\F) + 4\dim(L/Z_{n-1})\dim(L^n) - \dim(L^n) \\ &\leq \dim \cH^2(L/L',\F) + \dim(L'/L^n)[4\dim(B/B') - 1] \\ &~~~~~~~~~~~ + 4\dim(B/B')\dim(L^n) - \dim(L^n) \\ &= \dim \cH^2(L/L',\F) + [\dim(L') - \dim(L^n)][4\dim(B/B') - 1] \\ &~~~~~~~~~~~ + 4\dim(B/B')\dim(L^n) - \dim(L^n) \\ &= \dim \cH^2(L/L',\F) + \dim(L')[4\dim(B/B') - 1]
\end{align*} which yields the desired result.
\end{proof}

Noting that $\dim(B/B')\leq \dim(L/L')$, the next corollary is an immediate consequence of the previous one. What follows is an alternative way of writing our bound on $\dim \cH^2(L,\F)$ that is based on the dimensions of $L$ and $L/L'$.

\begin{cor}\label{yankosky 3.6}
$\dim \cH^2(L,\F) \leq \dim \cH^2(L/L',\F) + \dim(L')[4\dim(L/L') - 1]$.
\end{cor}

\begin{cor}
Let $n=\dim L$ and $d=\dim(L/L')$. Then \[\dim \cH^2(L,\F) \leq -2d^2 + d + 4nd - n.\]
\end{cor}

\begin{proof}
We first note that, since $L/L'$ is abelian, its multiplier $\cH^2(L/L',\F)$ has the maximal possible dimension of $2d^2$, a bound obtained in \cite{mainellis batten di}. Using Corollary \ref{yankosky 3.6}, we compute \begin{align*}
    \dim \cH^2(L,\F) &\leq 2d^2 + (n-d)[4d-1] \\ &= 2d^2+4nd - n - 4d^2 + d
\end{align*} since $\dim(L/L')=\dim L - \dim(L')$ implies that $\dim(L') = n-d$.
\end{proof}

\section{Associative Case}
We now consider the special case of associative algebras, as any associative algebra $L$ can be thought of as a diassociative algebra in which $x\dashv y = x\vdash y$. Indeed, this condition allows us to denote multiplication by $xy$ without distinction, and the axioms of the diassociative structure condense down to $x(yz) = (xy)z$. Next, consider a pair of associative algebras $A$ and $B$, and a central extension $0\xrightarrow{} A\xrightarrow{} L\xrightarrow{} B\xrightarrow{} 0$ of $A$ by $B$. A 2-cocycle $f\in \cZ^2(B,A)$ is a bilinear form $f:B\times B\xrightarrow{} A$ that satisfies $f(i,jk) = f(ij,k)$ for all $i,j,k\in B$ (see \cite{mainellis factor systems}). Continuing to follow the work of \cite{mainellis factor systems}, recall that any diassociative cocycle $(f\dd,f\vv)$ may be defined by a section $\mu:B\xrightarrow{} L_2$ of some equivalent extension. In particular, $f\dd(i,j) = \mu(i)\dashv\mu(j) - \mu(i\dashv j)$ and $f\vv(i,j) = \mu(i)\vdash\mu(j) - \mu(i\vdash j)$. In the associative case, one computes $f\dd(i,j) = f\vv(i,j)$, and we may thus think of our cocycle as a single bilinear form.

The five-term cohomological sequence is extended by $L/A\otimes B \oplus B\otimes L/A$ for the associative analogue of Theorem \ref{yankosky 2.1}, which need only require that $f(A,B) = 0$ and $f(B,A)=0$ for all cocycles $f\in \cZ^2(L,\F)$. Moreover, our $\delta$ map is defined by $\delta(f'+\cB^2(L,\F)) = (f'',g'')$, where $f'':L/A\times B\xrightarrow{} \F$ and $g'':B\times L/A\xrightarrow{} \F$. In the context of the diassociative to associative simplification, this pair would arise by computing equalities $f\dd'' = f\vv''$ and $g\dd'' = g\vv''$ via $f\dd'=f\vv'$. Similarly, our other results that extend the sequence by a term of the form $(X\otimes Y\oplus Y\otimes X)^2$ need only extend by the term $X\otimes Y \oplus Y\otimes X$ (as in the Leibniz sequences \cite{mainellis batten}). The associative analogue of Corollary \ref{yankosky 3.2} is thus the inequality \[\dim \cH^2(L,\F) + 1 \leq \dim \cH^2(L/Z,\F) + 2\dim(L/L')\] since $\dim(L/L'\otimes Z \oplus Z\otimes L/L') = 2\dim(L/L')$ in the case of $\dim Z =1$. The associative analogue of Corollary \ref{yankosky 3.4} is \[\dim \cH^2(L,\F) \leq \cH^2(L/L^n,\F) + 2\dim(L^n)\dim(L/Z_{n-1}) - \dim(L^n),\] that of Corollary \ref{yankosky 3.5} is \[\dim \cH^2(L,\F) \leq \dim \cH^2(L/L',\F) + \dim(L')[2\dim(L/Z(L)) - 2\dim((L/Z(L))') - 1],\] and that of Corollary \ref{yankosky 3.6} is \[\dim \cH^2(L,\F) \leq \dim \cH^2(L/L',\F) + \dim(L')[2\dim(L/L') - 1].\] For our last corollary, we obtain $\dim \cH^2(L,\F)\leq -d^2 + d + 2nd - n$ where $n=\dim L$ and $d=\dim(L/L')$.

\section{Example}
We denote by $M_{\As}(L)$ the multiplier of $L$ as an associative algebra, and by $M_{\Dias}(L)$ the same for $L$ as a diassociative algebra. Recall that $M_*(X) = \cH_*^2(X,\F)$, where $*$ ranges over the categories $\As$ and $\Dias$. As with the Leibniz multiplier, the dimension of the associative multiplier is bounded by $n^2$ for an algebra of dimension $n$ (by the same logic used in Lemmas 2.0.2 and 2.0.3 in \cite{rogers}). The dimension of the diassociative multiplier is bounded by $2n^2$ (see \cite{mainellis batten di}). These bounds are reached exactly when the algebra is abelian.

Let $L$ be the 2-dimensional associative algebra with basis $\{x_1,x\}$ and nonzero multiplication given solely by $x_1x_1 = x$.

\paragraph{Associative Extension.} We first compute the multiplier $M_{\As}(L)$ of $L$ as an associative algebra. Let $K$ be the cover $M\oplus L$ of $L$ with multiplications given by \begin{align*}
    &x_1x_1 = x+m_{11}, \\ & x_1x = m_{12}, \\ &xx_1 = m_{21}, \\ & xx = m_{22}.
\end{align*} To simplify, we let $x_2 = x+m_{11} = x_1x_1$, and thus multiplication in $K$ becomes \begin{align*}
    &x_1x_1 = x_2, \\ & x_1x_2 = m_{12}, \\ &x_2x_1 = m_{21}, \\ & x_2x_2 = m_{22}
\end{align*} where $M$ is generated by $m_{12},m_{21},m_{22}$ and $K$ is generated by $m_{12},m_{21},m_{22},x_1,x_2$. To find bases for our multiplier and cover, it remains to check linear relations between our generating elements. We note that any product of four or more elements in $K$ is zero, and, in particular, that $m_{22} = x_2x_2 = x_1x_1x_2 = x_1m_{12} = 0$. It thus suffices to plug our $x_1$'s into the associative identity. We compute \begin{align*}
    0 &= \As(x_1,x_1,x_1) \\ &= x_1(x_1x_1) - (x_1x_1)x_1 \\ &= x_1x_2 - x_2x_1 \\ &= m_{12} - m_{21}
\end{align*} which implies that $m_{12} = m_{21}$. We let $m_{12}\neq 0$ to obtain the maximal possible dimension of our defining pair, and thus $\{m_{12}\}$ forms a basis for $M = M_{\As}(L)$. In other words, $\dim M_{\As}(L) = 1$.

We now verify that the inequality \[\dim \cH^2(L,\F) + 1 \leq \dim \cH^2(L/Z,\F) + 2\dim(L/L')\] holds. Let $Z= \langle x\rangle$, noting that $Z \subseteq Z(L)\cap L'$ and $\dim Z = 1$. Since $L/Z$ is abelian, we know that $\dim M_{\As}(L/Z) = \dim(L/Z)^2 = 1$. Moreover, $\dim(L/L') = 1$, and thus the inequality is computed as $1 + 1 \leq 1 + 2(1)$, or $2\leq 3$. We can also check \[\dim \cH^2(L,\F) \leq \dim \cH^2(L/L',\F) + 2\dim(L')\dim(L/Z(L)) - \dim(L')\] for the associative analogue of Corollary \ref{yankosky 3.4}, since $L$ is nilpotent of class 2. Since $L' = Z(L) = \langle x\rangle$, we have $\dim M_{\As}(L/L') = (1)^2 = 1$ and $\dim(L') = \dim(L/Z(L)) = 1$. The inequality thus becomes $1\leq 1+2(1)(1) - 1$, or $1\leq 2$.

\paragraph{Diassociative Extension.} Our algebra $L$ can be thought of as the diassociative algebra with basis $\{ x_1,x\}$ and nonzero multiplications given solely by $x_1\dashv x_1 = x = x_1\vdash x_1$. Let $K$ be the cover $M\oplus L$ of $L$ with multiplications denoted by \begin{align*}
    &x_1\dashv x_1 = x+m_{11} && x_1\vdash x_1 = x+s_{11} \\ &x_1\dashv x = m_{12} && x_1\vdash x = s_{12} \\ &x\dashv x_1 = m_{21} && x\vdash x_1 = s_{21} \\ & x\dashv x = m_{22} && x\vdash x = s_{22}.
\end{align*} Letting $x_2 = x_1\dashv x_1 = x+m_{11}$, we obtain $x_1\vdash x_1 = x+s_{11} = x_2 - m_{11} + s_{11} = x_2+m$ for some $m\in M$. Thus, multiplication in $K$ is given by \begin{align*}
    &x_1\dashv x_1 = x_2 && x_1\vdash x_1 = x_2+m \\ &x_i\dashv x_j = m_{ij} && x_i\vdash x_j = s_{ij}
\end{align*} for $(i,j) \neq (1,1)$. Now $M$ is generated by $m,m_{12},m_{21},m_{22},s_{12},s_{21},s_{22}$. As in the associative case, we need to verify linear relations in $K$ based on the five axioms of diassociative algebras. Noting that any four-product is zero, we get $m_{22} = s_{22} = 0$, and it remains to plug $x_1$'s into the diassociative identities. We compute
\begin{align*}
    0&= \As\dd(x_1,x_1,x_1) \\ &= x_1\dashv(x_1\dashv x_1) - (x_1\dashv x_1)\dashv x_1 \\ &= x_1\dashv x_2 - x_2\dashv x_1 \\ &= m_{12} - m_{21}
\end{align*} and \begin{align*}
    0&= \As\vv(x_1,x_1,x_1) \\ &= x_1\vdash(x_1\vdash x_1) - (x_1\vdash x_1)\vdash x_1 \\ &= x_1\vdash (x_2+m) - (x_2+m)\vdash x_1 \\ &= s_{12} - s_{21}
\end{align*} which yields $m_{12} = m_{21}$ and $s_{12} = s_{21}$. In a similar fashion, axioms (1) and (3) yield these same equalities respectively. Finally, axiom (2) yields $m_{21} = s_{12}$, and thus $\{m,m_{12}\}$ forms a maximal basis for $M$. Therefore, $\dim M_{\Dias}(L) = 2$, which is notably different from $M_{\As}(L)$.

To verify Corollary \ref{yankosky 3.2}, let $Z= \langle x\rangle$, which is again 1-dimensional and falls in $Z(L)\cap L'$. Since $L/Z$ is abelian, we have $\dim M_{\Dias}(L/Z) = 2\dim(L/Z)^2 =2(1)^2 = 2$. Our inequality \[\dim \cH^2(L,\F) + 1 \leq \dim \cH^2(L/Z,\F) + 4\dim(L/L')\] is thus satisfied, with $2+1 \leq 2 + 4(1)$, or $3\leq 6$. For Corollary \ref{yankosky 3.4}, we want \[\dim \cH^2(L,\F) \leq \dim \cH^2(L/L',\F) + 4\dim(L')\dim(L/Z(L)) - \dim(L')\] since $L$ is nilpotent of class 2. Since $L' = Z(L) = \langle x\rangle$, we have $\dim M_{\Dias}(L/L') = 2(1)^2 = 2$ and $\dim(L') = \dim(L/Z(L)) = 1$. The desired inequality is thus $2\leq 2+4(1)(1) - 1$, or $2\leq 5$.

\section*{Acknowledgements}
The author would like to thank Ernest Stitzinger for the many helpful discussions.

\end{document}